\documentclass[reqno,12pt]{amsart}
\usepackage{mathtools}
\usepackage{amsmath}
\usepackage{amssymb,latexsym}
\usepackage{fullpage}
\usepackage{color}
\theoremstyle{plain}
\newtheorem{theorem}{Theorem}[section]
\newtheorem{corollary}{Corollary}[section]
\newtheorem{lemma}{Lemma}[section]
\newtheorem{proposition}{Proposition}[section]
\theoremstyle{remark}
\newtheorem{definition}{Definition}[section]
\newtheorem{remark}{Remark}[section]
\newtheorem{example}{Example}[section]

\numberwithin{equation}{section}

\newcommand{\bC}{{\mathbb C}}

\newcommand{\bZ}{{\mathbb Z}}
\newcommand{\bN}{{\mathbb N}}

\newcommand{\Bo}{\mathcal{B} }
\newcommand{\Fe}{\mathcal{F} }
\newcommand{\al}{\alpha }
\newcommand{\bet}{\beta }

\newcommand{\la}{\lambda}

\newcommand{\<}{\langle }

\newcommand{\fg}{\mathfrak g }
\newcommand{\fo}{\mathfrak o }

\newcommand{\fsp}{\mathfrak {sp} }

\renewcommand{\>}{\rangle}

\begin{document}
\title{Vertex operators arising from Jacobi\,-\,Trudi identities}
\author{Naihuan Jing}
\address{Department of Mathematics, North Carolina State University, Raleigh, NC 27695, USA}
\email{jing@math.ncsu.edu}
\author{Natasha Rozhkovskaya}
\address{Department of Mathematics, Kansas State University, Manhattan, KS 66502, USA}
\email{rozhkovs@math.ksu.edu}

\begin{abstract}
We  give an interpretation of  the boson-fermion correspondence as a direct consequence of  Jacobi\,--\,Trudi identity. This viewpoint enables us
 to construct  from a generalized version of the Jacobi\,--\,Trudi identity
the action of  Clifford algebra  on polynomial algebras that arrives as  analogues  of the algebra of symmetric functions. A  generalized Giambelli identity is also proved to follow from that identity.  As applications, we obtain explicit formulas for vertex operators corresponding to characters of the classical Lie algebras, shifted Schur functions, and generalized Schur
symmetric functions associated to linear recurrence relations.
\end{abstract}
\maketitle

\section{Introduction}

The classical boson-fermion correspondence establishes an isomorphism of two modules of the Heisenberg algebra $\mathcal H$. One of the main tools of this correspondence is the identification of each of the  graded  components of the bosonic Fock space, which  is a  polynomial algebra  $B^{(m)}=\bC[p_1,p_2,\dots]$,  with the algebra of symmetric functions in some variables. Then $p_i$'s  are interpreted as the power sum  symmetric functions. Thus,  the  properties of  symmetric functions  proved to be very useful for  applications  of the boson-fermion correspondence, and were used by many authors in various papers.

Besides the Heisenberg algebra, there are also other important  algebraic structures acting on the Fock space:  a Clifford algebra, the Virasoro algebra, and the infinite-dimensional  Lie algebra $gl_\infty$. Their actions are closely related to each other. In particular, given an action of the Clifford algebra, using the normal ordered product, one can immediately construct the action of the Heisenberg algebra, the Virasoro algebra and of $gl_\infty$ (see e.g.   \cite{Bom},  15.3, 16.3, 16.4). One should  also notice that the  boson-fermion correspondence  provides  an important  example of an isomorphism with lattice vertex algebras (see e.g. \cite {FLM,FBZ,K,Bom}, etc.).

The action of the Clifford algebra on the fermionic Fock space is given by contraction and wedge operators (see (\ref{f+}),\,(\ref{f-})), and the action  on the  bosonic  Fock space is given by so-called vertex operators (see  (\ref{psi2}),\,(\ref{psi21})).  Note that these vertex operators are written  as products of two  exponential  functions, one of which ``contains all  differentiations'', and the other one -- all ``multiplication  operators''. The existence of such decomposition is very important for  further applications, where normal ordered products of such operators are used.

One of the most prominent applications of vertex operators action of  the Clifford algebra is  the study of  solutions  of soliton  equations.
 The algebraic structure of solutions of soliton equations is described in    \cite{ DJKM, JM}.  Vertex operators  presentation of the action of  the Clifford  algebra on a function space allows one to  construct the action of the the group $\mathrm{GL}_\infty$ on the same space. The orbit of the vacuum vector  under this action is an infinite-dimensional Grassmannian manifold, where the defining equations are equivalent to the soliton equations.
  \\

In  this paper   we  highlight  the  core connection  of  the boson-fermion correspondence to the theory of symmetric functions. Based on two  simple observations, it  allows one not only  to simplify   some classical computations,  but to recover the constructions of the boson-fermion correspondence  through generalizations of symmetric functions and,  as an immediate  consequence, to obtain the action of  other related algebraic structures on the polynomial rings.

The first  observation is that the boson-fermion  correspondence   is equivalent to the statement of Jacobi\,--\,Trudi  identity (\ref{classjt}). This  identity  expresses Schur functions as determinants  of  matrices with  complete symmetric functions as the entries.  Exactly the Jacobi\,--\,Trudi   identity  provides the  action of fermions on the ring   of symmetric functions. It seems that this simple  fact is generally overlooked in the literature, though  several papers can be named, where  versions of Jacobi-Trudi  identities are applied: (see e.g. \cite{C1,C2} where variations of  Jacobi\,--\,Trudi  identities  are used for construction of  $\tau$-functions). Note that, as we show below,  the so-called Giambelli identity in a  general  form is   equivalent  to a general  Jacobi\,--\,Trudi identity statement.

The second observation is the following. Traditionally,  the graded  components of the  bosonic Fock space are  identified with a  polynomial  algebra  $B^{(m)}=\bC[p_1,p_2,\dots]$, where  generators $p_i$  are power sums. This choice of  generators of the ring of  symmetric functions    gives very  simple description  of the action of generators of the Heisenberg algebra as  multiplication and differentiation  operators (\ref{al}). However,  it has  considerable disadvantages as well.
First, many computations with these generators  are technically involved:
 in contrast to  complete symmetric functions or  elementary  symmetric functions,
power sums are not  elements of the  well-studied  basis  of Schur functions $s_\lambda$.
The second  disadvantage  is  crucial for  applications  to  generalizations:  some  generalizations  of symmetric functions do not have a  unique  natural analogue of power sums (see e.g. section 3.1 in  \cite{GR}  or (12.20) - (12.25) in  \cite {OO1}).

As it will be shown below,  some computations can be significantly simplified  and standardized  if one  considers complete (or elementary) symmetric   functions  as the  primary  set of generators of  the polynomial algebra  $B^{(m)}$. This  approach was already used for vertex operators acting on the ring of  classical symmetric  functions, here we recall just a few of these cases. In \cite{Zel}  so-called Bernstein operators are described, which also appear  in  example I.5.29  of   \cite{Md}. In the same manner examples    III.5.8, III.8.8  of  \cite{Md}  interpret  the results  on vertex operator presentation of Hall-Littlewood and Schur Q-functions (see \cite {J, J12} for the original construction). In \cite{FVP} the categorification of boson-fermion correspondence is  also  based on  presenting the action of Clifford algebra through complete and  elementary  symmetric functions as well as  their adjoint  operators.
\\

These two  observations motivate the  main idea and the goal  of  this note to develop a uniform approach to construct  vertex operators  for different  generalizations of  symmetric functions that naturally appear in representation theory. Given an  analogue of  a  Jacobi\,--\,Trudi identity for a variation of the algebra of  symmetric functions, one  immediately  writes  the action of the Clifford algebra on that  algebra (and, using the normal ordered product, the  action of the Heisenberg algebra, the Virasoro algebra and $gl_\infty$). In some cases   it is possible  to  go  further  and to decompose  vertex operators  as  a product of  a generating  function for  multiplication operators and of a generating function for adjoint operators.
\\

  We would like to remark that there are  numerous   generalizations of  symmetric functions in mathematics,  and many of those  enjoy analogues of the Jacobi-Trudi formula.  It would not be  possible to recall all  of them here, so we mention  just a couple of inspiring examples.    Many examples are contained in \cite{Md} (including the  variations  discussed in \cite{Md1}). But more examples can be treated using the method in our paper and
  we list a few as follows. Characters of simple classical  Lie algebras
enjoy an analogue of the Jacobi-Trudi identity  \cite{FH, JTClass2}, as well as double symmetric functions \cite{Mol-dob} and, in particular, shifted symmetric functions \cite{OO1}. The latter give eigenvalues of   the basis elements  of the center of the universal enveloping algebra of the Lie algebra $gl_n$.   An analogue of  the Jacobi\,--\,Trudi  identity, but not in a determinant form,  exists for Macdonald polynomials \cite{ML}. In  \cite{C2}, \cite {SV}  a family of generalized Schur functions  is considered, where  powers of $x^k$ in the definition of a Schur function are  replaced by  recursively defined  functions $\phi_k(x)$.  Cherednik-Bazhanov-Reshetikhin  theorem   is  a Jacobi-Trudi-like  description of a  family of quantum transfer matrices of a quantum  integrable spin chains model with  rational $\mathrm{GL}(N)$-invariant R-matrix \cite {C1, Baz, Cher}. This list can be continued.
Many of these examples (in particular,   generalizations   \cite{Md}, \cite{SV}, \cite{C2}) fit into  the unified construction that  we discuss in this paper,  and therefore, the action of Clifford algebra on these examples is  defined in a  straightforward way. In some cases  it is possible to  write  vertex operators  in simple  form,  similar to the classical case.
\\

  The  paper is  organized as follows.
  In Section 2 we review  classical boson-fermion correspondence associated to symmetric functions.  In Section 3  we use
  a generalized form  of Jacobi\,--\,Trudi identity  to define combinatorially   analogues of Schur functions,  including  complete  and  elementary symmetric functions. We prove that they  enjoy  properties similar to  properties of classical  symmetric functions, such as Newtons' identity and Giambelli identity.
In Section  4 we define  the  action  of  Clifford algebra on generalized symmetric functions and give  vertex operators presentations of these functions.
In the classical case  the action  of  generators of Clifford  algebra  can be  expressed through  `multiplication' and  `differentiation'  operators.  We introduce in Section 4 analogues of  adjoint operators and write similar decomposition  for   the action of Clifford algebra on  generalized symmetric functions. In Section 5
we apply constructions of Section 3 and 4 to write closed  formulas for vertex operators in   several known  examples:  universal characters of classical Lie algebras, shifted symmetric functions and   symmetric functions associated to linear recurrence relations.

\section{Boson-fermion correspondence}
Recall the  algebraic  construction of the classical boson-fermion  correspondence (see e.g. \cite{DJKM, JM, Bom}).

Let $V=\oplus _{j\in \bZ}\bC\, v_j$  be an infinite-dimensional complex vector space with a linear  basis $\{v_j\}_{j\in \bZ}$.
 Define $F^{(m)}$ ($m\in \bZ$)  to be a linear span of semi-infinite wedge products
$v_{i_m}\wedge v_{i_{m-1}}\wedge\dots  $ with the properties \\
(1)\quad  $i_m>i_{m-1}>\dots$,  \\
(2)\quad  $i_k=k$  for $k<<0$.\\
The monomial   of the form $|m\>=v_m\wedge v_{m-1}\wedge\dots $ is called  the $m$th vacuum vector.
The elements of $F^{(m)}$ are linear  combinations of monomials  $ v_I= v_{i_1}\wedge v_{i_2}\wedge\dots $ that are different from $| m\>$ only at finitely many places.

We define the fermionic  Fock space as  the graded space $ \Fe=\oplus_{m\in \bZ} F^{(m)}$.
The  Clifford algebra (fermions)  acts on the  space $\Fe$  by wedge  operators $\psi_k$  and contraction operators $\psi^*_k$ ($k\in \bZ)$.
More precisely, these operators are defined by
\begin{align}\label{f+}
\psi_k \,( v_{i_1}\wedge v_{i_2}\wedge\dots)= v_k\wedge v_{i_1}\wedge v_{i_2}\wedge\dots,
\end{align}
and
\begin{align}\label{f-}
\psi^*_k\, ( v_{i_1}\wedge v_{i_2}\wedge\dots)=
 \delta_{k, i_1 }v_{i_2}\wedge v_{i_3}\wedge\dots
- \delta_{k, i_2 }v_{i_1}\wedge v_{i_3}\wedge\dots
+\delta_{k, i_3 }v_{i_1}\wedge v_{i_2}\wedge\dots-\dots\quad.
\end{align}
Then  the following relations are satisfied:
\begin{align*}
\psi_k\psi^*_m+\psi^*_m\psi_k=\delta_{k,m},\quad
\psi_k\psi_m+\psi_m\psi_k=0,\quad
\psi^*_k\psi^*_m+\psi^*_m\psi^*_k=0.
\end{align*}
Combine the operators $\psi_k, \psi^*_k$  in generating functions  (formal distributions)
\begin{align}
\Psi(u)=\sum_{k\in \bZ} \psi_k u^{ k}\quad \text{ and } \quad  
\Psi^*(u)=\sum_{k\in \bZ} \psi^*_k u^{- k}. 
\end{align}
 Using  the normal ordered product, we introduce the  formal distribution
$$
\alpha(u)=:\Psi(u)\Psi^*(u):\, =\Psi(u)_+\Psi^*(u)-\Psi^*(u)\Psi(u)_-, 
$$
where by definition of the normal ordered product,
$$
\Psi(u)_+=\sum_{k\ge 1} \psi_ku^{k},\quad \Psi(u)_-=\sum_{k\le 0} \psi_ku^{k}. 
$$
It can be verified that the coefficients $\alpha_k$ of the formal distribution  $\alpha(u)=\sum \alpha_k u^{-k}$  and the central element $1$ satisfy  relations of the Heisenberg algebra $\mathcal {A}$ (see e.g. \cite{Bom}, 16.3):
$$
[1,\alpha_k]=0,\quad   [\alpha_k,\alpha_m] =m\delta _{m,-k}  \quad (k,m\in \bZ).
$$
In this way the Fock space $\Fe$  becomes  an $\mathcal {A}$-module.

There is also a natural action of  the Heisenberg algebra $\mathcal {A}$  on the  boson space, which is  the algebra  of polynomials $\mathcal B^{(m)}=z^m\bC[ p_1, p_2,\dots ]$:  
\begin{align}\label{al}
\alpha_n=\frac{\partial}{\partial_{p_n}}, \quad \alpha_{-n}=n p_n,\quad  \alpha_0=m\quad(  n\in \bN, \quad m\in \bZ).
\end{align}
The boson\,--\,fermion correspondence  identifies the spaces $\Bo^{(m)}$ and $ \Fe^{(m)}$ as  equivalent  $\mathcal{A}$-modules (see e.g.  \cite{DJKM, F,JM,Bom}).
The vacuum vector $z^m$  is identified with $|m\>=v_m\wedge v_{m-1}\wedge\dots$.
The exact correspondence relies   on the  interpretation of  $p_k$'s  as  symmetric functions. More  precisely, $p_k$  is interpreted as a   $k$-th (normalized) power sum in some  set of variables. Then  each graded  component  $\Bo^{(m)}$ is viewed as the  ring of symmetric functions, which is  known to be the ring of  polynomials in variables $p_k$'s. There is a linear  basis of the  ring of symmetric functions,
 consisting of  Schur  symmetric functions $z^m s_\lambda$.
Then the  linear basis of elements  $v_\lambda=(v_{\lambda_1+m}\wedge v_{\lambda_2+m-1}\wedge v_{\lambda_3+m-2}\dots  )$ of $ \Fe^{(m)}$, labeled by partitions $\lambda=(\lambda_1,\ge \lambda_2,\ge \dots,\ge \lambda_l\ge 0)$   corresponds to the linear basis  $z^m s_\lambda$ of $\Bo^{(m)}$  (see e.g. \cite{Bom}  Theorem 6.1).

\smallskip

Moreover, by  this  correspondence the action of  operators  $\psi_k,\psi^*_k$  on $\Fe$ is carried  to the action  on the graded space  $\Bo=\oplus  \Bo^{(m)}$, where it is   described by  generating functions $\Psi (u), \Psi^*(u)$, written in the so-called vertex operator form (\ref{psi2},\ref{psi21}). \bigskip

\section{General Jacobi\,--\,Trudi identity}\label{ident}

\subsection* {Jacobi-Trudi identity for  classical symmetric functions}
First  let us recall the statement of the original  Jacobi-Trudi identity.
For more details please refer e.g. to \cite{Md, AL}.
Recall that the   ring of  classical symmetric functions in the variables $ (x_1,x_2,\dots)$ is a polynomial ring in  the coherent variables $\{h_k\}_{k\in \bZ_{\ge 0}}$, which  are  complete symmetric functions defined by
\begin{align}\label{hclass}
h_k=\sum_{i_1\le i_2\le \dots\le i_k }x_{i_1}x_{i_2}\dots  x_{i_k},\quad h_0=1.
\end{align}
The ring of symmetric functions possesses  a linear basis of   Schur  functions $s_\lambda$,  labeled by partitions $\lambda=(\lambda_1,\ge \lambda_2,\ge \dots,\ge \lambda_l\ge 0)$. It is  known that Schur symmetric  functions can be expressed as polynomials in the $h_k$ by  the Jacobi\,-\,Trudi  formula
\begin{align}\label{classjt}
s_\lambda =\det [h_{\lambda_i-i+j}]_{\{1\le i,j\le l\}}.
\end{align}

\subsection*{Generalization}
 Our goal is to  define  the action of Clifford algebra  on rings of  different analogues of  symmetric functions and, whenever it is  possible,  to  write  closed formulas  for vertex operators.   A generalized version of Jacobi-Trudi identity  will be the main tool  for that.
 While  in the classical  setting Jacobi-Trudi identity is a theorem that describes  a property of  Schur symmetric functions,  for our purposes firsthand  we want to consider
the elements  {\it defined} by a generalized version  of Jacobi-Trudi identity. Under the condition  of   linear independence of these elements,
this general setting  immediately implies  the action of Clifford algebra on their linear span.
Then  this general setting  can be applied to particular examples as follows. Given a variation of symmetric functions with an established analogue of Jacobi-Trudi identity, one can match  that identity as a  particular case of (\ref{jt1}) with  appropriate interpretation of  $h^{(p)}_k$'s  and consider the  action  (\ref{p10}, \ref{p100}) of Clifford  algebra.  Formulas (\ref{vk22}, \ref{vk11}) express the action of  fermions through  multiplication operators and corresponding adjoint operators,  and in some cases it is possible  go  further and  to deduce   simple  formulas for the generating functions  (vertex operators).  The examples of such applications are provided in  Section 5.

\smallskip

We start with a set of independent  variables  $\{h^{(0)}_k\}_{k\in \bZ_{k> 0}}$.    Let $B= \bC[h^{(0)}_1,  h^{(0)}_2,\dots ]$  be  a polynomial ring in these variables.  Let  
$$
h^{(0)}_{0}=1, \quad\text{and}\quad  h^{(0)}_{k}=0 \quad \text{for}\quad k<0.
$$
The algebra $B$ is graded by $deg(h_n^{(0)})=n$, so
\begin{equation}
B=\bigoplus_{n=0}^{\infty} B_n,
\end{equation}
where $B_n$ is the homogeneous subspace of degree $n$ and $B_0=\mathbb C$. Let $B^{\leq n}=\oplus_{i=0}^{n}B_{n-i}$.
Then $B=\bigcup_{n=0}^{\infty}B^{\leq n}$ and clearly $gr\, B\simeq B$.
Suppose that  $\{h^{(r)}_k\}$ $(k,r\in \bZ)$   is a  set of elements of $B$ such that
$h_k^{(r)}\in B^{\leq k+r}$. 
We {\it require }that  
\begin{align}\label{prop11}
h^{(k)}_{-k}=1, \quad \text{and} \quad h^{(r)}_{k}=0 \quad \text{for}\quad k+r<0.
\end{align}
For any  partition $\lambda=(\lambda_1\ge \lambda_2\ge \dots \ge \lambda_l\ge 0)$ the  polynomial $s_\lambda$ in   $\{h^{(0)}_k\}_{k\in \bZ}$  is {\it defined} by the formula
\begin{align}\label{jt1}
s_\lambda =\det \left[h_{\lambda_i-i+1}^{(j-1)}\right]_{\{1\le i,j\le l\}}.
\end{align}


We also {\it require}  that  
polynomials  $\{s_{\lambda}\}$ form a linear basis of $ B=\bC[h^{(0)}_1, h^{(0)}_2, \dots]$  when  $\lambda$  goes over the set of all partitions. In the following we will frequently refer to $s_{\la}$ as the {\it generalized Schur function} associated to the partition $\lambda$.

Moreover, we can extend the definition of  polynomials  $s_{\lambda}\in B$   and define  such a polynomial for any   integer vector $\lambda= (\lambda_1, \lambda_2,\dots, \lambda_l)$  -- it does not have to be a partition or even a
composition,  parts $\lambda_i$ can be negative,  and can be listed in  any  order.  We  represent the integer vector $\lambda$ as  a sequence with finitely many non-zero terms:  $\lambda= (\lambda_1, \lambda_2,\dots \lambda_l, 0,0,0\dots )$.
 Given  such $\lambda$, we associate the infinite matrix
\begin{align*}
S_\lambda=
\begin{pmatrix}
h^{(0)}_{\lambda_1} &h^{(1)}_{\lambda_1} &h^{(2)}_{\lambda_1}  &\dots&h^{(l-1)}_{\lambda_1}&| &\dots\\
h^{(0)}_{\lambda_2-1} &h^{(1)}_{\lambda_2-1} &h^{(2)}_{\lambda_2-1}&\dots &h^{(l-1)}_{\lambda_2-1}&| &\dots\\
\dots &\dots&\dots&\dots&\dots&|&\dots\\
h^{(0)}_{\lambda_l-l+1} &h^{(1)}_{\lambda_l-l+1} &h^{(2)}_{\lambda_l-l+1} &\dots&h^{(l-1)}_{\lambda_l-l+1}&| &\dots\\
\\
\hline\\
h^{(0)}_{-l} &h^{(1)}_{-l} &h^{(2)}_{-l}&\dots &h^{(l-1)}_{-l} &|&\dots\\
h^{(0)}_{-l-1} &h^{(1)}_{-l-1} &h^{(2)}_{-l-1}&\dots &h^{(l-1)}_{-l-1} &|&\dots\\
\dots &\dots&\dots&\dots&\dots&|&\dots
\end{pmatrix}
\end{align*}
Note that   by property (\ref{prop11}),  $S_\lambda$  has  a  block form
$$
\begin{pmatrix}
A&B\\
0&D
\end{pmatrix}
$$
 with $A$  being an  $l\times l$ matrix and $D$ --  an infinite upper-triangular  matrix with one's on the diagonal. Therefore, if we consider  the sequence of determinants of the $N\times N$ parts of the matrix $S_\lambda$
 $$
\det \left[h_{\lambda_i-i+1}^{(j-1)}\right]_{i,j=1,\dots, N}, \quad N=1,2,3,\dots ,
 $$
 for  large enough $N$  it will stabilize to a polynomial in  variables $h^{(0)}_1, h^{(0)}_2,\dots $, which we will  still denote as $s_\lambda$ or as $\det S_\lambda$.

 \begin{lemma}  For any integer vector $\lambda$ the following properties  hold:

\begin{align}\label{p2}
 A)\quad s_{(\lambda_1, \lambda_2,\dots, \,\lambda_k, \lambda_{k+1},\,\dots, \lambda_l)}= - s_{(\lambda_1, \lambda_2,\dots, \, \lambda_{k+1}-1, \lambda_{k}+1,\,\dots, \lambda_l)}
\end{align}
\begin{align} \label{p3}
B) \quad \text{
If   $\lambda_k- \lambda _m= k-m$, for some $k,m$, then $s_\lambda=0$.}
\end{align}
\end{lemma}
\begin{proof}
Change of the order of rows in the determinant  $\det S_\lambda$.
\end{proof}
\begin{remark}The case of  classical symmetric functions correspond to  $h^{(0)}_k=h_k$ -- the  ordinary complete symmetric functions,  and $
h^{(p)}_k=h_{k+p}.$
\end{remark}

\subsection*{Elementary symmetric functions }
There are several ways to define elementary symmetric functions. Traditionally, classical elementary symmetric functions are defined by the formula
\begin{align}\label{elem}
e_k=\sum_{i_1< i_2< \dots< i_k }x_{i_1}x_{i_2}\dots  x_{i_k},\quad e_0=1.
\end{align}
Elementary symmetric functions enjoy several important properties that are equivalent to definition  (\ref{elem}).
For our purposes  it is convenient to  consider the  so-called
{\it Newton's identity} as the  primary  defining property of elementary symmetric functions.
The classical Newton's formula  relates  classical  elementary symmetric functions $e_i$ to   complete  symmetric functions $h_k$:
\begin{align}\label{New_clas}
\sum_{i=0}^{n}(-1)^ie_i h_{n-i}=0 \quad \text{for $n\ge 1$},\quad\quad   e_0h_0=1.
\end{align}
This  identity is also equivalent to the condition that for any $n\in \bZ_{\ge 0}$, the matrix\\ $E=((-1)^{i-j}e_{i-j})_{\{i,j=1,\dots,n\}}$ is the inverse  of  the matrix $H=(h_{i-j})_{\{i,j=1,\dots n\}}$.
The formula (\ref{New_clas}) can be  considered as one of  alternative  definitions of elementary  symmetric functions, since   (\ref{New_clas})  implies that  $e_k$'s are uniquely  defined  polynomials in $h_l$'s  given by
\begin{align}\label{e_def}
e_k= \det[h_{1-i+j}]_{1\le i,j\le k}.
\end{align}
 In our case we  generalize the formula (\ref{e_def})   to become  {\it the  definition} of  elementary symmetric functions.

\begin{definition}\label{def_el}
For $a,p\in \bZ$ define the {\it generalized elementary symmetric function} $e^{(p)}_a$ as follows:
\begin{align*}
e^{(p)}_{a}=\begin{cases}0,\quad  \text {for}\quad  p< a,\\
1,\quad\text {for}\quad  p= a, \\
\det \left[h_{p+1-i}^{(-p+j)}\right]_{1\le i,j\le {p-a}},\quad  \text {for}\quad  p>a.
\end{cases}
\end{align*}
\end{definition}

Thus,  for $p>a$, the polynomial $e^{(p)}_{a}$ is a $(p-a)\times (p-a)$ determinant

\begin{align*}
e^{(p)}_a
=\det \begin{pmatrix}
h_{p}^{(-p+1)} &h^{(-p+2)}_{p} &\dots &h^{(-a-1)}_{p}&h^{(-a)}_{p}\\
1&h^{(-p+2)}_{p-1} &\dots&h^{(-a-1)}_{p-1}&h^{(-a)}_{p-1}\\
\dots&\dots&\dots&\dots\\
0&0&\dots&1&h^{(-a)}_{a+1}
\end{pmatrix}.
\end{align*}
\begin{remark}
Comparing the definition with  formula  (\ref{jt1}), note that
$
e^{(1)}_{-a}= s_{(1^{a+1})}
$, if $a\geq -1$.
Also  $e^{(a+1)}_a= h^{(-a)}_{a+1}$. \end{remark}
\begin{proposition}\label{p:newton}
The elements   $e^{(p)}_{a}$  are  solutions of the following Newton's identity:
\begin{align}\label{Ng}
\sum_{p=-\infty}^{\infty}(-1)^{a-p}  h^{(p)}_{b}e^{(-p)}_{a} =\delta_{a,b} \quad  \text{for any $a,b \in \bZ$.}
\end{align}
\end{proposition}
\begin{proof} For $b>a$ expand the determinant  in the definition of  $e^{(b)}_{a}$ by the  first  row:
\begin{align*}
e^{(b)}_a&= h^{(-b+1)}_b e^{(b-1)}_a-h^{(-b+2)}_b e^{(b-2)}_a+\dots  +(-1)^{b-a-1}h^{(-a)}_b e^{(a)}_a,
\end{align*}
which gives $\sum_{s=0}^{b-a}(-1)^{s}h^{(-b+s)}_be^{(b-s)}_a=0$.
Since $h^{(-b+s)}_b=0$ for $s<0$ and $e^{(b-s)}_a=0$ for $s>b-a$,
we can rewrite the last  equality
$$\sum_{p=-\infty}^{\infty}(-1)^{a-p}h^{(p)}_be^{(-p)}_a=0\quad \text{for $b>a$}.$$
It is easy to get the same equality  for $b<a$, since in this case (\ref{Ng}) contains only zero terms. Finally, for $a=b$,
\begin{align*}
\sum_{p=-\infty}^{\infty}(-1)^{a-p}h^{(p)}_ae^{(-p)}_a= h^{(-a)}_ae^{(a)}_a=1.
\end{align*}
 \end{proof}
Let us introduce  the infinite matrices    $\mathcal H$ and
$\mathcal E$ with entries
\begin{align}\label{HE}
 \mathcal H_{bp} = h_{-b}^{(p)},
\quad
\mathcal E_{pa} = (-1)^{a-p}e_{-a}^{(-p)}, \quad (a,b,p\in\bZ).
\end{align}
They can be displayed as follows.
\begin{align*}
\mathcal H
=\begin{pmatrix}
\dots&\dots&\dots&\dots&\dots&\dots&\dots &\dots\\
\dots&h_{2}^{(-1)}&h_{2}^{(0)} &h^{(1)}_{2} &\dots &h^{(a-1)}_{2}&h^{(a)}_{2}&\dots\\
\dots&1&h_{1}^{(0)} &h^{(1)}_{1} &\dots &h^{(a-1)}_{1}&h^{(a)}_{1}&\dots\\
\dots&0&1&h^{(1)}_{0} &\dots&h^{(a-1)}_{0}&h^{(a)}_{0}&\dots\\
\dots&\dots&\dots&\dots&\dots&\dots&\dots\\
\dots&0&0&0&\dots&1&h^{(a)}_{-a+1}&\dots\\
\dots&\dots&\dots&\dots&\dots&\dots&\dots&\dots
\end{pmatrix},
\end{align*}

\begin{align*}
 \mathcal E=
 \begin{pmatrix}
  \dots& \dots& \dots& \dots& \dots &\dots&\dots&\dots\\
      \dots&1&-e_{0}^{(1)}& + e_{-1}^{(1)}&- e_{-2}^{(1)}&\dots& \pm e_{-a}^{(1)}&\dots\\
     \dots&0&1& - e_{-1}^{(0)}&+ e_{-2}^{(0)}&\dots &\pm e_{-a}^{(0)}&\dots\\
      \dots&0&0&1&- e_{-2}^{(-1)}&\dots&\pm  e_{-a}^{(-1)}&\dots\\
       \dots&0&0& 0&1&\dots& e_{-a}^{(-2)}&\dots\\
       \dots& \dots& \dots& \dots& \dots&\dots&\dots&\dots
 \end{pmatrix},
\end{align*}
where the columns are numerated to grow from the left to the right and rows are indexed to increase from the top to the bottom.
Both  matrices are upper-triangular with one's on the diagonal.
Then (\ref{Ng})  can be interpreted as  $\mathcal H\mathcal E=Id$ for infinite  matrices, or, more carefully, for any $M<N$,  for finite  upper-triangular submatrices
$\mathcal H(M,N)= (\mathcal H_{bp})_{(M\le b,p\le N)}$, and $\mathcal E (M,N)= (\mathcal E_{qa})_{(M\le q,a\le N)}$:
\begin{align*}
\sum_{k=M}^{N}\mathcal H_{bk}\mathcal E_{ka}=\sum_{s=0}^{a-b}\mathcal H_{b,b+s}\mathcal E_{b+s,a}=
\sum_{s=0}^{a-b}(-1)^{-s+a-b}h_{-b}^{(b+s)} e_{-a}^{(-b-s)}= (-1)^{a-b}\delta _{-a,-b}=\delta _{a,b}.
\end{align*}
Hence for any $M<N$ one has
$
\mathcal H(M,N)\mathcal E(M,N)= Id.
$

The following lemma is proved in \cite{FH} (Lemma A.42).
\begin{lemma} \label{let}Let $A$ and $B$  be $r\times r$ matrices whose product is a scalar matrix $c\cdot Id$. Let $(S,S^\prime)$ and $(T,T^\prime)$ be permutations of the sequence $(1,\dots, r)$, where  $S$ and $T$ each  consists of $k$ integers, $S^\prime$, $T^\prime$  of $r-k$.  Denote  as $A_{S,T}$ the  corresponding minor (it is the  determinant of the $k\times k$ matrix  whose $i,j$  entry is $a_{s_i, t_j}$ with
$s_i\in S, t_j\in T$).
Then
$$
c^{r-k}A_{S,T}=\varepsilon \det(A) B_{T^\prime, S^\prime},
$$
where $\varepsilon $ is the product of the signs  of the  two permutations.
\end{lemma}

 \begin{corollary}
 Let $\lambda=(\lambda_1,\dots,\lambda_l)$ be a partition, and let $\mu=(\mu_1,\dots \mu_k)$  be the conjugate partition.
 Then
 \begin{align*}
 s_\lambda= \det[e_{j-\mu_j}^{(i)}]_{1\le i,j\le k}. 
 \end{align*}
 \end{corollary}
 \begin{proof}
From (\ref{jt1}), $s_\lambda$  is the minor $\mathcal H(-k,l-1)_{S,T}$ of the matrix   $\mathcal H(-k,l-1)$ with
 $S=\{-\lambda_1,\dots,-\lambda_l+l-1\}$ and $T=\{ 0,1,\dots, l-1\}$. Then
 $S^\prime=\{\mu_1-1, \dots, \mu_k-k\}$,  $T^\prime=\{-k,\dots, -1\}$, and
 $ 
 S\sqcup S^\prime = T\sqcup T^\prime=\{-k,\dots,l-1\}$, with
$  \varepsilon= (-1)^{\sum\mu_j}= (-1)^{\sum\lambda_j}.
$
 By Lemma \ref {let} it follows that
\begin{align*}
s_\lambda&=\det \mathcal H(-k,l-1)_{S,T}=\varepsilon  \det \mathcal E(-k,l-1)_{T^\prime,S^\prime}=\varepsilon  \det[(-1)^{\mu_j-j-i}e_{j-\mu_j}^{(i)}]_{1\le i,j\le k}\\
&=\varepsilon  (-1)^{\sum_j (\mu_j-j)-\sum_i i}\det[e_{j-\mu_j}^{(i)}]_{1\le i,j\le k}=\det[e_{j-\mu_j}^{(i)}]_{1\le i,j\le k}.
\end{align*}
 \end{proof}

\subsection*{General Giambelli identities}
Recall that sometimes the linear   basis of the fermionic Fock space is matched with the linear basis of symmetric  functions not through the Jacobi-Trudi identity, but through another (equivalent)  combinatorial  property of classical  Schur functions, the  so-called Giambelli  identity (see e.g.  (1.18) in \cite{JM} or (2.10) in \cite{C1}). Here we prove the analogue of  Giambelli identity  generalized  symmetric functions defined by (\ref{jt1}).

 For any $m, n\in \mathbb Z$, we define the hook Schur function $s_{(m|n)}$
by
\begin{equation}\label{e:shook}
s_{(m|n)}=\sum_{p=0}^n(-1)^ph_{m+1}^{(p)}e_{-n}^{(-p)}.
\end{equation}
If $m\geq 0$ and $n\geq 0$, $s_{(m|n)}$ is exactly the generalized Schur function $s_{(m+1,1^n)}$
of the hook $(m+1, 1^n)$ according to (\ref{jt1}). This can be easily seen by expanding the
Jacobi-Trudi determinant along the first row and observe that
the $(1, p+1)$-minor is exactly the elementary symmetric function
$e_{-n}^{(-p-1)}$ by Definition \ref{def_el}, for $p=0, 1, \cdots, n$.

Clearly by Newton identity (\ref{Ng}),
$s_{(m|n)}=0$ if either $m<0$ or $n<0$, except that $s_{(m|n)}=(-1)^n$ when $m+n=-1$.

Recall that the Frobenius notation $(\al_1\cdots \al_r|\bet_1\cdots \bet_r)=(\al|\bet)$ of the partition $\lambda$ is defined by
\begin{align*}
\al_i&=\lambda_i-i,\\
\bet_i&=\lambda_i'-i,
\end{align*}
for $i=1, \cdots, r$, where $r$ is the length of the main diagonal in $\lambda$. Then  the conjugate of $(\al|\bet)$ is $(\bet|\al)$, and the hook $(m+1,1^n)$ is $(m|n)$ in Frobenius notation.
Sometimes it is convenient to allow $i>r$ and still use the formula to extend the Frobenius notation.
For example, $\lambda=(3\;2^3)=(2\;0|3\;2)$ in Frobenius notation
and $\la=(2,0,-1,-2,\cdots|3,2,1,0,-1,\cdots)$ or any cut-off beyond $r$ in the extended Frobenius notation.

\begin{theorem} For any partition $\la=(\al|\bet)=(\al_1\cdots \al_r|\bet_1\cdots \bet_r)$
and any $n\geq l(\lambda)$, the generalized Schur function $s_{(\al|\bet)}$ satisfies
\begin{equation}\label{e:giambelli}
s_{(\al|\bet)}=\det[s_{(\la_i-i|n-j)}]_{1\leq i, j\leq n}=\det[s_{(\al_i|\bet_j)}]_{1\leq i, j\leq r}.
\end{equation}
\end{theorem}
\begin{proof}
 Let $(\al|\bet)=\lambda=(\lambda_1,\cdots, \lambda_l)$. For any $n\geq l(\la)$, consider the matrix
 \begin{equation*}
 [(-1)^{i-1}e_{-n+j}^{(-i+1)}]_{1\leq i, j\leq n}.
 \end{equation*}
  As $e_{-n+j}^{(-i+1)}=0$ for $i+j>n+1$, the
matrix has determinant $1$ regardless of the parity of $n$. By the definition of the hook Schur function (\ref{e:shook}) it follows that
\begin{equation}
s_{(\la_i-i|n-j)}=\sum_{p=0}^{n-1}(-1)^{p}h_{\la_i+1-i}^{(p)}e_{-n+j}^{(-p)}.
\end{equation}
This implies the following matrix identity in $Mat_{n}(B)$:
\begin{equation}\label{e:giambelli2}
[s_{(\la_i-i|n-j)}]=[h_{\la_i-i+1}^{(j-1)}][(-1)^{i-1}e_{-n+j}^{(-i+1)}].
\end{equation}
Taking the determinant of (\ref{e:giambelli2}) we get that
$$s_{\la}=\det[s_{(\la_i-i|n-j)}]_{1\leq i, j\leq n}.$$

For each $i>r$, the $i$th row has only one non-zero entry $s_{(\la_i-i|n-j)}=(-1)^{n-j}$ at the column $j=n+1+\al_i$.
It is well-known that $\{n+1+\al_i\}_{1\leq i\leq n}\sqcup\{n-\bet_j\}_{1\leq j\leq m}=
\{1, 2\cdots, m+n\}$ for any $m\geq \la_1$ (cf. \cite{Md}, I, (1.7)), i.e. $\{-1-\al_i\}_{1\leq i\leq n}\sqcup\{\bet_j\}_{1\leq j\leq m}=
\{-m, \cdots, n-1\}$. So when we remove the last $n-r$ rows and the $n-r$ columns
numbered by $n+1+\al_{r+1}, \cdots, n+1+\al_n$, the supplement
$r\times r$-minor of $[s_{(\la_i-i|n-j)}]$ is exactly $\det[s_{(\al_i|\bet_j)}]_{r\times r}$.
Therefore $s_{\lambda}=\pm \det[s_{(\al_i|\bet_j)}]$. The overall sign factor
is given by
$$
(-1)^{\sum_{i=r+1}^n (n-j_i)+((i+j_i)-(i-r-1))}=(-1)^{\sum_{i=r+1}^n(n+r+1)}=(-1)^{(n-r)(n+r+1)}=1,
$$
which shows the Giambelli identity.
\end{proof}
\begin{remark}
It is clear that the Jacobi\,-\,Trudi identities for both $h_n^{(i)}$ and $e_{-n}^{(i)}$ are
special cases of the Giambelli identity.

For example, let  $\la=(3\;2^3)=(2\;0|3\;2)$, so $r=2$. Take $n=4$ and
the extended Giambelli determinant is
$$\det[s_{(\al_i|4-j)}]=det\begin{bmatrix} s_{(2|3)} & s_{(2|2)}& s_{(2|1)}& s_{(2|0)}\\
s_{(0|3)} & s_{(0|2)}& s_{(0|1)}& s_{(0|0)}\\
0 & 0 & 0&1\\
0 &0& -1& 0
\end{bmatrix}=\det[s_{(\al_i|\bet_j)}]_{2\times 2}.$$

\end{remark}

\section{Clifford algebra action  and Vertex Operator Presentation.} 
  In accordance with  the classical case, to construct the action of  the Clifford algebra on the boson space, we need to take
  multiple  copies of the  polynomial  ring  $B$. Namely, let $z$ be a variable,
  set  $\Bo^{(m)}=z^m B$ and $\Bo=\oplus\Bo^{(m)}= \bC[z, z^{-1}, h^{(0)}_1, h^{(0)}_2, \dots]$.
  The elements  $\{s_\lambda z^m\}$  form a linear basis of $\Bo$, where   $s_\lambda $ are defined
  by (\ref{jt1}) and are labeled by  partitions  $\lambda$, and  $m\in \bZ$. 

Define the  operators $\psi_{k}$ and $\psi^*_{k}$   ($k\in \bZ$) acting on this basis  by the following rules:
\begin{align}
\psi_{k} (s_\lambda z^m) &= s_{(k-m-1,\lambda)} z^{m+1},\label{p10}\\ 
\psi^*_k (s_\lambda z^m)&=\sum_{t=1}^{\infty}(-1)^{t+1}\delta_{k-m-1,\,\lambda_t-t}\,  s_{(\lambda_1+1,\dots, \lambda_{t-1}+1, \lambda_{t+1},\lambda_{t+2}\dots) } z^{m-1}.  \label{p100} 
\end{align}
Note that in the sum (\ref{p100}) only one term survives by the property (\ref{p3}) of $s_\lambda $. 
 A direct check  on the basis elements $\{s_\lambda z^m\}$ immediately proves  the following  proposition:
\begin{proposition}\label{prop1}
(1)  The action of  ${\psi_k, \psi^*_k}$  satisfies  the commutation relations of the Clifford algebra of fermions:
\begin{align*}
\psi_{k} \psi_{l}+  \psi_{l} \psi_{k}= 0,
 \quad \psi^*_{k}\psi^*_{l}+  \psi^*_{l} {\psi^*}_{k}= 0,
\quad \psi_{k}\psi^*_{l}+  \psi^*_{l} \psi_{k}= \delta_{k,l}.
\end{align*}
 (2) Let $\lambda=(\lambda_1\ge \lambda_2\ge\dots \ge \lambda_l)$ be a partition $\lambda\vdash|\lambda| $, and let  $\mu=(\mu_1\ge \mu_2\ge\dots \ge \mu_k)$ be the conjugate of the partition $\lambda$. Then
\begin{align}
&\psi_{\lambda_{l}+l}\dots \psi_{\lambda_{2}+2}\psi_{\lambda_{1}+1}  (1)= s_\lambda z^l,\label{bers1}
\\
&\psi^*_{-\lambda_{1}-l+1}\dots \psi^*_{-\lambda_{l-1}-1} \psi^*_{-\lambda_{l}}  (1)= (-1)^{|\lambda|}s_\mu z^{-l}.\label{bers2}
\end{align}
\end{proposition}
\begin{remark}
 Following \cite{Zel}, where (\ref{bers1}) was stated for classical symmetric functions,  sometimes
 presentation of this kind  are called
   ``Bernstein  (vertex) operators presentation''.  Vertex operators  and  the  presentation  of classical symmetric functions (\ref{bers1}), (\ref{bers2})
were constructed in  \cite{J}.
\end{remark}

Our next goal is to combine the operators $\psi_k, \psi^*_k$ into generating functions and write them in the form  of ``vertex operators''.
 Set
  \begin{align}\label{Gamma}
\Psi(u,m)=\sum_{k\in \bZ}\psi_{k}\,\vert_{B^m}u^{k},\quad 
\Psi^*(u,m)=\sum_{k\in \bZ}\psi^*_{k}\,|_{B^m}u^{-k}.
\end{align}

\subsection*{ Vertex operators for classical symmetric functions}
First, let us review   the combinatorial  description of  operators  (\ref{Gamma}) in the classical case.
Recall  (see e.g. I. 5, \cite{Md}) that the ring of symmetric functions  possess  a scalar product with  classical   Schur  functions $s_\lambda$  as an orthonormal basis:
              $\<s_\lambda,s_\mu\>=\delta_{\lambda,\mu}$. Then for any symmetric function $f$ one can define an adjoint operator  acting on the ring of symmetric functions  by the  standard rule:
        $
        \<D_f g, w\>= \<g, fw\>
        $
        for any symmetric functions  $g, f,w$.
 Using  notation
        $D_\lambda:= D_{s_\lambda}$, one gets
\begin{align*}
        \< D_\lambda s_\mu,s_\nu\>= \< s_\mu,s_\lambda s_\nu\>,\quad\text{hence}\quad   D_\lambda s_\mu= s_{\mu/\lambda},
\end{align*}
where $s_{\mu/\lambda}$  is the skew-Schur function of shape $\mu/\la$.  This  symmetric function can be  expressed as a determinant
\begin{align}
s_{\mu/\lambda}=\det [h_{\mu_i-\lambda_j-i+j}]_{1\le i,j\le n}.\label{skew}
\end{align}
Introduce    generating functions for  $e_k$, $h_k$  and for corresponding adjoint operators:
\begin{align*}
E(u)= \sum_{k\ge 0} {e_k} u^k,\quad H(u)= \sum_{k\ge 0} {h_k} u^k,
\\
  DE(u)= \sum_{k\ge 0} D_{e_k} u^k,\quad DH(u)= \sum_{k\ge 0} D_{h_k} u^k.
  \end{align*}
     Then it is known that for classical  symmetric functions
\begin{align}
    \Psi(u,m)&= u^{m+1}z \,H(u) DE\left(\frac{-1}{u}\right),\label{psi1}\\
       \Psi ^*(u,m)&= u^{{-m}}z^{-1} \, E(-u) DH\left(\frac{1}{u}\right).\label{psi11}
\end{align}
Note that  very often  these formulas  are  written through power sums  (see e.g. \cite{Bom}, Lecture 5). Namely, introduce the (normalized) classical power sums, which are   symmetric functions  of the form
\begin{align}\label{class}
p_k=\frac{1}{k}\sum_{i} x_{i}^k.
\end{align}
Then
 $H(u)=\exp(\sum_{k\ge 1}p_ku^k)$ 
 and
$E(u)=\exp(-\sum_{k\ge 1} p_k (-u)^k)$, since $H(u)E(-u)=1$.

Any  symmetric function  $f$ can  be expressed as a polynomial   $ f=\varphi(p_1,2p_2, 3p_3,\dots)$ in (normalized) power sums. Then one has
$D_f=\varphi ({\partial_{ p_1}},  {\partial_ {p_2}}, {\partial_{ p_3}},\dots)$.
 (See e.g. \cite {Md}, I.5, example 3).
Hence
$$
DH(u)=\exp\left(\sum_k {\frac {\partial_{ p_k}}{k}}u^k\right), \quad  DE(u)= \exp\left(-\sum_k \frac{\partial_{ p_k}}{k}(-u)^k\right),
$$
and we can write
\begin{align}
     \Psi(u,m)= u^{m+1}z\exp \left(\sum_{j\ge 1}p_j u^j  \right) \exp \left(-\sum_{j\ge 1}\frac{ \partial_{p_j}}{j} {u^{-j}}\right),
\label{psi2}\\
     \Psi^*(u,m)= u^{-m}z^{-1} \exp \left(-\sum_{j\ge 1} {p_j} u^j  \right) \exp \left(\sum_{j\ge 1}\frac{\partial_{p_j}}{j}  {u^{-j}}\right).
\label{psi21}
\end{align}
Note that the formulas  for vertex operators has the  form of  decomposition into a  product of two generating functions,  which  separate the action of   differentiations, and of  multiplication operators. This decomposition is important for  further  applications of  vertex operators.

\subsection*{Operators $\psi_{k},\psi^*_k$ expressed through $D_{p}$ and $D^{(p)}$}
 Our  next goal is  to   express operators  $\psi_k, \psi^*_k$ through  analogues of  $e^{(p)}_k$, $h^{(p)}_k$ $D_{e_k}$, $D_{h_k}$.
In some examples we will be  able to go further and  write nice formulas  in the spirit of  (\ref{psi1}) and (\ref{psi11}) for generating functions $\Psi (u,m), \Psi^*(u,m)$.

Motivated by  definition (\ref{skew}) of classical   skew-symmetric functions  through determinants, we introduce two operators on $B$, which are generalizations of
$D_{h_p}$ and $D_{e_p}$.
Define
\begin{align}\label{Dp}
D_{p} (s_{\lambda_1,\dots, \lambda_l}):=
\begin{cases}
\det [h^{(j)}_{\lambda_i-i}]_{(j=0,1,\dots, \hat p, \dots l,\quad  i=1,\dots,l)}, & \text{if} \quad  0\le p\le l,\\ 
 0,&\text{otherwise}.
\end {cases}
\end{align}
The  following interpretation of this definition   will be useful. For $0\le p\le l$, one can write
\begin{align}\label{DP1}
D_{p} (s_{\lambda_1,\dots, \lambda_l})=
\det
\begin{pmatrix}
h_{\lambda_1-1}^{(0)} &\dots &h^{(p-1)}_{\lambda_1-1} &h^{(p+1)}_{\lambda_1-1} &\dots&h^{(l)}_{\lambda_1-1}&| &\dots\\
h_{\lambda_2-2}^{(0)} &\dots  &h^{(p-1)}_{\lambda_2-2} &h^{(p+1)}_{\lambda_2-2}&\dots &h^{(l)}_{\lambda_2-2}&| &\dots\\
\dots &\dots&\dots&\dots&\dots&\dots&|&\dots\\
h_{\lambda_l-l} ^{(0)}&\dots &h^{(p-1)}_{\lambda_l-l}&h^{(p+1)}_{\lambda_l-l} &\dots&h^{(l)}_{\lambda_l-l}&| &\dots\\
\\
\hline\\
h_{-l-1}^{(0)} &\dots &h^{(p-1)}_{-l-1} &h^{(p+1)}_{-l-1}&\dots &h^{(l)}_{-l-1} &|&\dots\\
h_{-l-2}^{(0)} &\dots &h^{(p-1)}_{-l-2}  &h^{(p+1)}_{-l-2}&\dots &h^{(l)}_{-l-2} &|&\dots\\
\dots &\dots&\dots&\dots&\dots&\dots&|&\dots
\end{pmatrix}.
\end{align}
The matrix  above is obtained from the matrix  $S_\lambda$ by   subtracting  $1$ from  all  the lower indices in the entries of  $S_\lambda$ and  deleting  the $p$-th  column.
Note that the  resulting matrix has the form
$$
\begin{pmatrix}
A&B\\
0 &D
\end{pmatrix},
$$
where $A$ is $(l-1)\times (l-1)$ matrix,  and  $D$  is upper-triangular matrix with $1$'s  on the diagonal.
Note that this  interpretation naturally  extends to the case  $p>l$, where $D$  block  would have several zero's on the diagonal.
We can define a sequence of polynomials in $h_k^{(a)}$ -- the determinants of $N\times N$  left  upper  part of that matrix. This sequence  stabilizes for big enough $N$, and we conclude that
 presented by (\ref{DP1}) determinant  $D_{p} (s_{\lambda_1,\dots, \lambda_l})$ is a well-defined polynomial in $h^{(a)}_k$. Note that $D_0(1)=1$ and $D_p(1)=0$ for $p\ne 0$.
Operator $D_{p}$ is  the combinatorial analogue of the classical  operator $D_{e_p}$. 

Define
\begin{align}\label{Dpu}
D^{(p)}(s_{\lambda_1,\dots, \lambda_l })=\sum_{t=1}^{\infty} (-1)^{t+1}h^{(p)}_{\lambda_t-t+2} s_{(\lambda_1+1,\dots, \lambda_{t-1}+1, \lambda_{t+1},\dots ,\lambda_l,0,0,\dots)}.
\end{align}
This also can be  written as
\begin{align*}
D^{{(p)}} (s_{\lambda_1,\dots, \lambda_l}):=
\det
\begin{pmatrix}
h_{\lambda_1+1} ^{(p)} &h^{(0)}_{\lambda_1+1}&\dots&h^{(l-1)}_{\lambda_1+1}&| &\dots\\
h_{\lambda_2}^{(p)}  &h^{(0)}_{\lambda_2} &\dots  &h^{(l-1)}_{\lambda_2}&| &\dots\\
\dots &\dots&\dots&\dots&|&\dots\\
h_{\lambda_l-l+2} ^{(p)} &h^{(0)}_{\lambda_l-l+2} &\dots&h^{(l-1)}_{\lambda_l-l+2}&| &\dots\\
\\
\hline\\
h_{-l+1}^{(p)} &h^{(0)}_{-l+1}&\dots &h^{(l)}_{-l+1} &|&\dots\\
h_{-l}^{(p)} &h^{(0)}_{-l}&\dots &h^{(l)}_{-l} &|&\dots\\
&\dots&\dots&\dots&|&\dots
\end{pmatrix} .
\end{align*}

The  matrix  above  is obtained from the matrix $S_\lambda$  by raising all  the lower indices of its elements  by $1$ and   by  adding  the  first  column  with entries  $(h^{(p)}_{\lambda_i-i+2})_{i=1,2,\dots }$.
Then  this  matrix  has the form
$$
\begin{pmatrix}
A&B\\
C&D
\end{pmatrix},
$$
where $A$ is $(l+1)$ by  $(l+1)$ matrix,  $C$  is a matrix with all  zero-entires except (may be) the first column, and  $D$  is upper-triangular matrix with $1$'s  on the diagonal. Note that
$
h_{-i+2}^{(p)} =0$  for  $-i+2+p<0$, and therefore, the  sequence of determinants of $N\times N$ upper left corner submatrices of this matrix stabilizes  to a well-defined  polynomial in variables
$h^{(k)}_a$. Hence  $D^{(p)}$  is a well-defined operator  on $B$.

Note that  $D^{(p)}\equiv 0$  for $p\ge 0$,     since in this case the defining matrix  in $D^{(p)} (s_{\lambda_1,\dots, \lambda_l})$ will have two identical columns.
Also note that $D^{(p)} (s_{\lambda_1,\dots, \lambda_l})=0$ if $-p>\lambda_1+1$,   $D^{(-1)}(1)=1$, and $D^{(p)}(1)=0$ for $p\ne-1$.
Thus, $D^{(-p)}$ is  the analogue of $D_{h_{p-1}}$.

 \begin{proposition}\label{p_1}
 \begin{align}
  \psi_{k}|_{B^m}&=\sum_{p\in \bZ} (-1)^p h_{k-m-1}^{(p)} D_{p}z,\label{vk22}\\
 \psi^*_{k}|_{B^m} &=(-1)^{k-m+1} \sum_{p\in \bZ} (-1)^{p}e^{(p)}_{k-m+1} D^{(-p)}z^{-1} \label{vk11}.
 \end{align}
  \end{proposition}

 \begin{proof} We compare the action of operators on  both sides on the basis elements $s_\lambda z^m$.
 For the proof of (\ref{vk22}) use the expansion of the determinant $s_{(k-m-1,\lambda)}$ by the first  row:
 \begin{align*}
   \psi_{k} (s_\lambda z^m)&=s_{(k-m-1,\lambda)} z^{m+1}=  \sum_{p=0}^{l} (-1)^ph_{k-m-1}^{(p)} \det [h^{(j)}_{\lambda_i-i}]_{(j=0,1,\dots, \hat p, \dots l,\,  i=1,\dots,l)} z^{m+1} \\
    &=\left(\sum_{p\in \bZ}(-1)^ph_{k-m-1}^{(p)} D_{p} z \right) s_{\lambda}z^m.
     \end{align*}
For the proof of (\ref{vk11}) we  apply Newton's formula: 
 \begin{align*}
 &\left(\sum_{p\in \bZ} (-1)^{p+k-m+1}e^{(p)}_{k-m+1} D^{(-p)} z^{-1}\right)s_\lambda z^m \\
  &= \sum_{p\in \bZ} (-1)^{p+k-m+1}e^{(p)}_{k-m+1}\sum_{t=1}^{\infty}(-1)^{t+1} h^{(-p)}_{\lambda_t-t+2} s_{(\lambda_1+1,\dots, \lambda_{t-1}+1,  \lambda_{t+1},\lambda_{t+2}, \dots )} z^{m-1}
 \\
 &= \sum_{t=1}^{\infty}(-1)^{t+1}  \left(\sum_{p\in \bZ} (-1)^{p+k-m+1}e^{(p)}_{k-m+1}\,h^{(-p)}_{\lambda_t-t+2}\right)s_{(\lambda_1+1,\dots, \lambda_{t-1}+1,  \lambda_{t+1},\lambda_{t+2}, \dots )}z^{m-1}\\
  &=\sum_{t=1}^{\infty}(-1)^{t+1}  \delta_{k-m+1, \lambda_t-t+2}\,s_{(\lambda_1+1,\dots, \lambda_{t-1}+1,  \lambda_{t+1},\lambda_{t+2}, \dots )}z^{m-1}= \psi^*_{k} (s_\lambda z^m).
 \end{align*}

\end{proof}

\section{Applications - explicit formulas for vertex operators}

In this section we illustrate on examples  how the generalized Jacobi-Trudi identity  almost  effortlessly  provides vertex operators for the action of Clifford algebra on  some analogues of symmetric functions.
Note also  that, given  formal  distributions $\Psi (u)$, $\Psi^*(u)$ those coefficients  satisfy Clifford algebra relations, one can   define
the formal distributions
\begin{align*}
\alpha(u)=:\Psi(u)\Psi^*(u):\,,\quad
X(u,v)=:\Psi(u)\Psi^*(v):\,,\\
L(u)={1}/{2}( :\partial\Psi(u)\Psi^*(u):+:\partial \Psi^*(u)\Psi(u):),\quad
\end{align*}
with the property that the coefficients of $\alpha(u)$  satisfy relations of Heisenberg algebra  generators, coefficients  of $X(u,v)$ satisfy the relations of  Lie  algebra $gl_\infty$  generators, and  $L(u)$  is a  Virasoro   formal distribution (see e.g.  (16.11), (16.15) and  Corollary 16.1 in  \cite{Bom}).
\subsection*{Characters of classical  Lie algebras }
Let $\fg$ be a symplectic or orthogonal Lie  algebra. In  \cite{JTClass2} it is shown that, similarly to Schur functions,  universal characters  $\{\chi_{\fg}(\lambda)\}$ of  irreducible $\fg$-representations
form  a linear basis of  the ring of  ordinary symmetric functions.  For $a>0$ denote as  $J_a$   the universal  character of irreducible $\fg$-representation  with the highest weight $a\omega_1$, where $\omega_1$
is the first fundamental weight. Set $J_0=1$, and   $J_a=0$ if $a<0$.
Several forms of Jacobi-Trudi identities for characters of irreducible representations  of classical Lie algebras are known (e.g. \cite{FH}, \cite{JTClass2}).
Here we would like to use the identities of     \cite{FH} (Propostions 24.22, 24.33,  24.44):
\begin{proposition}
Let $\lambda=(\lambda_1\ge \lambda_2\ge\dots \ge  \lambda_r>0)$ be a partition.
Then
the universal  character $\chi_{\fg}(\lambda)$ of  irreducible $\fg$-representation  with the highest weight $ \lambda$
is given by
\begin{align*}
\chi_{\fg}(\lambda)=\det[{h^{(j-1)}_{\lambda_i-i+1}}]_{i,j=1\dots r},
\end{align*}
\begin{align}\label{ch1}
h^{(r)}_{a}=\begin{cases}
J_{a+r}+J_{a-r},\quad \text{if} \quad r> 0,\\
J_{a+r},\quad \text{if} \quad r\le 0.
\end{cases}
\end{align}
\end{proposition}

Hence we consider the boson space  with   $\{h^{(r)}_a\}$  as in (\ref{ch1}), and  the linear basis consisting of elements  ${s_\lambda= \chi_\fg(\lambda)}$.
This allows to write the corresponding  variation of formulas (\ref{psi1}), (\ref {psi11}).
\begin{proposition}\label{char}
Generating  functions (\ref{Gamma}) are given by
\begin{align*}
\Psi(u,m)&=u^{m+1}zJ(u) ( DE(-u)+DE(-1/u)- D_0), \\
\Psi^*(u,m)&= u^{-m+1}z^{-1}J(u)^{-1}(DH(u)- DH(1/u)),
 \end{align*}
 where $J(u)=\sum_{s\in \bZ}   { J_{s}} {u^{s}} $, 
   $DE(u)=\sum_{p\in \bZ}   { D_{p}} {u^{p}},$   and  $DH(u)=\sum_{p\in \bZ}   { D^{(p)}} {u^{p}}$, with
 operators $D_{p}$ and  $D^{(p)}$
  defined by (\ref{Dp}),  (\ref{Dpu}).
\end{proposition}
\begin{proof}
Using that $D_p\equiv 0$ for  $p<0$, we derive
 \begin{align*}
  \Psi(u,m)&
=u^{m+1}z\sum_{s\in \bZ}   \sum_{p\ge 0}  (-1)^p h^{(p)}_{s} D_p\,  {u^s} \\
&=u^{m+1}z\sum_{s\in \bZ} \big( \sum_{p\ge 0}  (-1)^p (J_{s+p}+ J_{s-p})  D_p\,-\,J_sD_0 \big){u^s} \\
&=u^{m+1}zJ(u) (DE(-1/u)+DE(-u)- D_0).
\end{align*}
Let $K(-u)=\sum_{p\in \bZ}K_p (-u)^{p}$  be  a formal series defined by the  property
\begin{align}\label{JK1}
J(u)K(-u)=1.
\end{align}
 As usual, (\ref{JK1})  is equivalent to  the  matrix form  relation
$
\mathcal J\mathcal K=Id,
$
where $ \mathcal{J}= \sum_{i,j\in \bZ} J_{-i+j}E_{ij}$,  and  $\mathcal K= \sum_{i,j\in \bZ}(-1)^{j-i} K_{j-i}E_{ij}$.
Define the  the matrix   $A=\sum_{i\in\bZ} E_{ii}+\sum_{i>0} E_{-ii}$ and the  inverse
$A^{-1}=\sum_{i\in\bZ} E_{ii}-\sum_{i>0}E_{-ii}$. Then  the matrix $\mathcal H$  in (\ref{HE})   for  $h^{(r)}_a$  from (\ref{ch1}) is  given by    $\mathcal H= {\mathcal{J}}A$, and,
since $\mathcal H \mathcal E=Id$,
 we have $\mathcal E=A^{-1} {\mathcal{K}}$ and
 \begin{align*}
e^{(i)}_{j}=
\begin{cases}
K_{i-j}, &\quad
 \text {for}\quad  i\le 0,\\
K_{i-j}-K_{-i-j},&\quad \text {for}\quad  i>0.
\end{cases}
\end{align*}

Using that $D^{(-p)}\equiv 0$ for  $p\le 0$, we derive
 \begin{align*}
  \Psi^*(u,m)
&=u^{-m+1}z^{-1} \sum_{s\in \bZ}   \sum_{p>0}  (-1)^p e^{(p)}_{s} {(-u)^{-s}}D^{(-p)} \\
&=u^{-m+1}z^{-1} \sum_{s\in \bZ}    \sum_{p>0}  (-1)^p (K_{p-s}-K_{-p-s})(-u)^{-s} D^{(-p)} \\
&=u^{-m+1} z^{-1}    K(-u) \sum_{p>0}  (-1)^p ((-u)^{-p} -(- u)^p) D^{(-p)}  \\
&=u^{-m+1}   z^{-1} K(-u) (DH(u) - DH(1/u)).
\end{align*}
\end{proof}
 \begin{remark}  Relation (\ref{JK1}) implies  that, together   with
$K_0=1$, and
$K_p=0$  for $p<0$,
$K_{p}= \det[ J_{1-i+j}] _{1\le i,j\le p}$, for $ p>0$, and $\chi_\fg((1^{p}))= K_p-K_{p-2}$.
\end{remark}
 \begin{remark}
 It is known that characters of   the classical Lie  algebras can be  expressed through evaluations of classical symmetric  functions, where
 the following identifications take place:
 \begin{align*}
J _a&= h_a(x_1,\dots, x_n, x_1^{-1},\dots x_n^{-n})-h_{a-2}(x_1,\dots, x_n, x_1^{-1},\dots x_n^{-n})\quad \text{for }\quad \fg= \fo_{2n},\\
J_a&= h_a(x_1,\dots, x_n, x_1^{-1},\dots x_n^{-n},1)-h_{a-2}(x_1,\dots, x_n, x_1^{-1},\dots x_n^{-n},1)\quad \text{for}\quad  \fg= \fo_{2n+1},\\
J_a&= h_a(x_1,\dots, x_n, x_1^{-1},\dots x_n^{-n}) \quad \text{for} \quad \fg= \fsp_{2n},
 \end{align*}
where  $h_k$  are  the ordinary complete symmetric functions (\ref{hclass}).
Accordingly,  formulas of  Proposition \ref{char} can be rewritten in terms of the power sums  (\ref{class}), and then
 one can compare  Proposition \ref{char} with  the  vertex operators introduced  in \cite{NJ}.  In \cite{NJ}  the opposite  direction is  undertaken:  vertex operators that  give realization of characters of classical Lie algebras are  introduced  through their action  on the boson space. The  operators satisfy generalized  fermions relations, this allows the  authors of \cite{NJ} to deduce  several versions of Jacobi-Trudi  identities for  characters of  the orthogonal and symplectic Lie algebras.
\end{remark}

\subsection*{Shifted Schur functions.}
We follow   notations  and  definitions of \cite{OO1}.  Shifted  symmetric functions are well-known for their  applications in the study of the centers  of  universal enveloping algebras,  of Capelli-type identities and   of asymptotic  characters  for unitary groups and symmetric groups. Let $U(gl_n)$  be the universal enveloping  algebra  of the  Lie algebra $gl_n$. The Harish-Chandra isomorphism  identifies  the center of  $U(gl_n)$  with the algebra of  shifted symmetric functions,  sending  a central element to its eigenvalue on a highest  weight module.
There is a distinguished  basis of the center, the images of the elements of that basis  under the isomorphism are the  so-called  shifted Schur functions.
Thus, using Harish-Chandra isomorphism, the constructed below action of Clifford algebra  on shifted symmetric functions can be  transported to the action of Clifford algebra on the center of $U(gl_n)$.

Combinatorially  a shifted Schur polynomial   $s^*_\lambda(x_1,\dots, x_n)$ can  be defined as
 a ratio of determinants
\begin{align}
s_\lambda^*(x_1,\dots, x_n)=\frac{\det(x_i+n-i|\lambda_j+n-j)}{\det(x_i+n-i|n-j)},
\end{align}
where
\begin{align*}
 (x|k)= 
 \begin{cases}
 x(x-1)\dots (x-k+1) \quad \text{for  $k=1,2\dots,$}\\
 1,\quad \text{for  $k=0$,}\\
  \frac{1}{(x+1)\dots (x+(-k) )} \quad \text{for  $k=-1,-2\dots.$}\\
 \end{cases}
 \end{align*}
Note that for generic  value of $x$  we can write $(x|k)= \frac{{\bf \Gamma}(x+1)}{{\bf \Gamma}(x+1-k)}$, where ${\bf\Gamma} (x)$ is the special gamma-function.
Also the following relations are useful in further  calculations:
\begin{align} \label{pr1}
(x|k)\,(-x-1|-k)=(-1)^k,\quad\quad  {(x|a)}
={(x|b)} {(x-b|a-b)}.
\end{align}
 The stability property of shifted Schur  polynomials allows  to introduce the shifted Schur functions, which we denote as
$s_\lambda^*=s_\lambda^*(x_1, x_2,\dots)$.
In  particular, the  complete  shifted Schur functions $h^*_r=s^*_{(r)}$ are
\begin{align*}
h^*_r(x_1,x_2,\dots )=\sum_{1\le i_1 \le \dots \le i_r <\infty} (x_{i_1}-r+1)(x_{i_2}-r+2)\dots x_{i_r},
\end{align*}
and the  elementary shifted Schur functions $e^*_r=s^*_{(1^r)}$ are
\begin{align*}
e^*_r(x_1,x_2,\dots )=\sum_{1\le i_1 <\dots < i_r <\infty} (x_{i_1}+r-1)(x_{i_2}+r-2)\dots x_{i_r}.
\end{align*}
 Theorem 13.1  in  \cite{OO1} states that
\begin{align*}
s^*_\lambda=\det[\phi^{j-1}h^*_{\lambda_i-i+j}]_{1\le i,j\le l},
\end{align*}
where $\phi$ is the automorphism of  the algebra  of  shifted Schur  functions, defined by the formula
$$ \phi(h^*_k)=h^*_k+(k-1)h^*_{k-1}.$$

By Corollay 1.6. in \cite{OO1},  shifted Schur  functions $s^*_\lambda$  form a linear basis in the ring of shifted symmetric functions, which is also a polynomial  ring in variables
$h^*_1, h^*_2,\dots$. Thus, we  are exactly in the setting of the  Section \ref{ident} with $s^*_\lambda=\det[{h^{(j-1)}_{\lambda_i-i+1}}]$ and
$$
h^{(r)}_k=\phi^r h^*_{k+r}. 
$$

In this case our  Definition \ref{def_el}  gives for $p>a$,
\begin{align}\label{sh2}
e^{(p)}_a=\det[\phi^{-p+j}h_{1-i+j}]_{1\le i,j\le p-a}.
\end{align}
Note that  $
e^*_k=  e^{(1)}_{1-k}.
$
Moreover,\begin{align*}
\phi^{-p+1}(e^*_k)=\det[\phi^{j-1}(\phi^{-p+1}h_{1-i+j})]_{1\le i,j\le k}=\det[\phi^{-p+j}h_{1-i+j}]_{1\le i,j\le k}= e^{(p)}_{p-k}.
\end{align*}

Traditionally,  for shifted Schur functions  generating functions  are written  not in powers  $u^k$, but in  shifted powers ${(u|k)}$, which makes some  formulas easier to work with (see  e.g. \cite {OO1}).
Let us  consider both cases: as before, we define    $\Psi(u,m)$, $\Psi^{*}(u,m)$   by (\ref{Gamma}), using  ordinary powers of $u$.
 Alternatively,
set
\begin{align*}
 \tilde \Psi(u,m)= \sum_{k\in \bZ}{\psi}_{k}|_{B_m}
\frac {1} {(u|k-1)},
\quad
 \tilde \Psi^*(u,m)
 = \sum_{k\in \bZ}{\psi}^{*}_{k}|_{B_m} {(-u|k+1)},
\end{align*}
using  shifted powers of $u$.
\begin{proposition} Generating functions  have the   form
\begin{align*}
 \Psi(1/u, m )&
 = zu^{-m-1} \sum_{p\ge 0} (\partial_u-1)^{p}   \big ( H(1/u)\big) \,\,
 D_p{u^p},\\
 \Psi^*(1/u,m)&
=  z^{-1}u^{m}\sum_{p\ge 0} (\partial_u+1)^{p}  \big(E(-1/u)\big) \,D^{(-p-1)} u^{p},\\
\tilde  \Psi(u,m)&=z\frac{1}{(u|m)} \tilde H(u-m) {D\tilde E}(u-m),\\
\tilde \Psi^*(u,m)&= z^{-1} {(-u|m)}\tilde E(u+m)   {D\tilde H}(u+m-1),
\end{align*}
where
\begin{align*}
E(u)&=\sum_{r\in \bZ}  e^{*}_{r}{u^{r}},\quad
  H(u)=\sum_{r\in \bZ} h^{*}_{r} {u^{r}},\quad  \tilde E(u)= \sum_{r\in \bZ} \frac{e^{*}_{r}}{(u|r)},\quad \tilde H(u)= \sum_{r\in \bZ}    \frac{h^{*}_{r}}{(u|r)},\\
 {D\tilde E}(u)&=\sum_{p\in \bZ}D_p\frac{(-1)^p}{(u|-p)}, \quad {D\tilde H}(u)=\sum_{p\in \bZ} D^{(p)}  \frac{(-1)^p }{(u|p)}.
\end{align*}
\end{proposition}
\begin{proof}
For the proof of the  first equality  write:
\begin{align*}
 \Psi(1/u,m)&
=u^{-m-1}z\sum_{s\in \bZ}   \left( \sum_{p\ge 0}   (-1)^p \phi^p h^{*}_{s+p}D_p  \right) {u^s}\\
&= u^{-m-1}z  \sum_{p\ge 0} \left( \sum_{i=0}^p \sum_{s\in \bZ} {p\choose i} (s+p-1|i)h^{*}_{s+p-i}   {u^{-s}}\right) (-1)^pD_p \\
&=  u^{-m-1}z \sum_{p\ge 0} \left( \sum_{i=0}^p \sum_{r\in \bZ} {p\choose i} (r+i-1|i)h^{*}_{r}   {u^{-r+p-i}}\right) (-1)^pD_p \\
&= u^{-m-1}z\sum_{r\in \bZ}  h^{*}_{r}  \sum_{p\ge 0}\left( \sum_{i=0}^p  {p\choose i} (r+i-1|i) {u^{-r-i}}\right) (-1)^pD_p{u^p}.
\end{align*}
Using that
$(-\partial_u)^i(u^{-r})
=(r+i-1|i)u^{-r-i}$,
we obtain
\begin{align*}   \sum_{i=0}^{p}  {{p}\choose i} (r+i-1|i)  {u^{-r-i}}= \left(\sum_{i=0}^{p}  {{p}\choose i} (-\partial_u)^i \right) ({u^{-r}})= (1-\partial_u)^{p}({u^{-r}}),
\end{align*}
and
\begin{align*}
\Psi(1/u, m)
&=  u^{-m-1}z \sum_{p\ge 0}  (1-\partial_u)^{p}   \left( \sum_{r\in \bZ} h^{*}_{r} {u^{-r}} \right)(-1)^pD_p{u^p}
 =u^{-m-1}z   \sum_{p\ge 0} (\partial_u-1)^{p}   \big ( H(1/u)\big) \,\, D_p{u^p}.
\end{align*}
Similarly,
\begin{align*}
 \Psi^*(1/u,m)
 &=u^{m-1}z^{-1}\sum_{s\in \bZ}   \left( \sum_{p> 0}  (-1)^p \phi^{-p+1} e^{*}_{p-s}{(-u)^s}\,D^{(-p)}  \right) \\
&=  u^{m-1}z^{-1} \sum_{r\in \bZ}  e^{*}_{r}  \sum_{p> 0}\left( \sum_{i=0}^{p-1}  {{p-1}\choose i} (r+i-1|i)  {(-u)^{-r-i}}\right) D^{(-p)}{u^{p}}\\
&=  u^{m-1}z^{-1}  \sum_{p\ge 0}  (1+\partial_u)^{p-1}   \left(\sum_{r\in \bZ}  e^{*}_{r}{(-u)^{-r}} \right) D^{(-p)} {u^{p}}\\
&=  u^{m-1}z^{-1} \sum_{p\ge 0} (1+\partial_u)^{p-1}  (  E(-1/u)) \,D^{(-p)} u^{p}.
\end{align*}
For the proof of the third equality write,  using (\ref{pr1}),
\begin{align*}
\tilde \Psi(u,m)
&=z\frac{1}{(u|m)}\sum_{k\in \bZ}   \left( \sum_{p\ge 0}   (-1)^p h^{(p)}_{k-m-1}D_p  \right)\frac{1} {(u-m|k-m-1)}\\
&=z\frac{1}{(u|m)}\sum_{s\in \bZ}   \left( \sum_{p\ge 0}   (-1)^p \phi^p h^{*}_{s+p}D_p  \right) \frac{1} {(u-m|s)}\\
&= z \frac{1}{(u|m)}  \sum_{p\ge 0} \left( \sum_{i=0}^p \sum_{s\in \bZ} {p\choose i} (s+p-1|i)h^{*}_{s+p-i} \frac{1}{(u-m|s)}\right) (-1)^pD_p \\
&=z\frac{1}{(u|m)}\sum_{r\in \bZ}  h^{*}_{r}   \sum_{p\ge 0} \left( \sum_{i=0}^p  {p\choose i} \frac{(r+i-1|i) }{(u-m|r-p+i)}\right) (-1)^pD_p.
\end{align*}

Using that  for $v\notin \bZ_{\le 0}$  and  $r,t\in\bZ$, $p\in \bZ_{>0}$,
 \begin{align*}
 \sum_{k=0}^p {p\choose k} \frac{(r+k-1|k) }{(v|t-p+k)}
 = \frac{{\bf\Gamma}(v-t+1) {\bf\Gamma}(v-t+p+r+1)}{{\bf \Gamma}(v-t+r+1){\bf \Gamma}(v+1)}=\frac{1}{(v-t+r|r)(v|t-p-r)},
\end{align*}

we derive the third statement:
\begin{align*}
\tilde  \Psi(u,m)
&= z\frac{1}{(u|m)} \sum_{r\in \bZ}   \frac{ h^{*}_{r}}{(u-m|r)}\sum_{p\ge 0}{\frac{(-1)^pD_p}{(u-m|-p)}}.
\end{align*}
Similarly,
\begin{align*}
\tilde \Psi^*(u,m)
&=z^{-1}{(-u|m)} \sum_{k\in \bZ}  \sum_{p> 0} e^{(p)}_{k-m+1}\, { (-1)^{k-m+1} (-u-m|k-m+1)}\, (-1)^pD^{(-p)}\\
&=z^{-1}{(-u|m)}\sum_{s\in \bZ}   \sum_{p> 0} \phi^{-p+1} (e^{*}_{p-s})\,  {(-1)^s} {(-u-m|s)} \,(-1)^pD^{(-p)}  \\
&=z^{-1}{(-u|m)}\sum_{s\in \bZ}   \left( \sum_{p> 0}  \sum_{i=0}^{p-1}  {{p-1}\choose i}  \frac{(p-s-1|i)} {(u+m-1|-s)}e^{*}_{p-s-i}(-1)^{p}D^{(-p)}  \right)\\
&= z^{-1}{(-u|m)} \sum_{r\in \bZ}  e^{*}_{r}  \sum_{p> 0}\left( \sum_{i=0}^{p-1}  {{p-1}\choose i}  \,  \frac{(r+i-1|i)} {(u+m-1|r+i-p)}\right) (-1)^{p}D^{(-p)}\\
&= z^{-1}{(-u|m)} \sum_{r\in \bZ}  \sum_{p> 0}     \frac{e^{*}_{r}\,(-1)^{p}D^{(-p)} } {(u+m|r)(u+m-1|-p)}\\
&= z^{-1} {(-u|m)} \sum_{r\in \bZ}   \tilde E(u+m) D\tilde H(u+m-1).
\end{align*}
and the fourth statement  is proved.
\end{proof}

\subsection*{Linear recurrence }
Assume that  elements $h^{(p)}_k$ of $B$ are given by a  linear recurrence relation
\begin{align}\label{rec2}
h^{(p)}_k= \sum_{i} t_{ki} h^{(p-1)}_{i},\quad  t_{ki}\in \bC, \quad k,i\in\bZ.
\end{align}
Note that  condition (\ref{prop1}) implies that $t_{k\,k+1}=1$ and $t_{k\,m}=0$ for $m>k+1$. 
We  combine the  coefficients  into an infinite matrix $T=\sum_{kl}t_{-k\,-l} E_{kl}$.
Classical symmetric functions are a particular case of this example.

Even though  the Clifford algebra action is  automatically defined   by (\ref{p10}), (\ref{p100}), a  recursion of type (\ref{rec2}) may not lead in general   to  a  desired decomposition of  formal distributions $\Psi(u,m), \Psi^{*}(u,m)$ as a  product of
 generating functions  of the type $H(u)$, $DH(u)$, $E(u)$, $DE(u)$. In this case one may want   to substitute the relation on generating functions by  a relation  on generating matrices and vectors.

 Namely,  let $\mathcal H$  and  $\mathcal E$  be infinite  generating matrices, defined  as  in
(\ref{HE}). Let ${\mathcal  H}\,^{(p)}$ be the $p$-th column of $\mathcal H$, and  let  ${\mathcal  E}\,^{(p)}$ be the the $(-p)$-th row of $\mathcal E$.
Note that
$ {\mathcal  H}\,^{(p)}=T{\mathcal  H}\,^{(p-1)}$  and the property  $\mathcal H\mathcal E=\mathcal E\mathcal H= Id$  implies that
${\mathcal  E}\,^{(p)}T={\mathcal  E}\,^{(p-1)}$, and

\begin{align*}
e^{(p)}_k=\sum_{i}(-1)^{k+i+1}e^{(p-1)}_i t_{i\,k}.
\end{align*}

Define the generating vectors $\overline{\Psi}(m), \overline{\Psi}^{*}(m)$ for $ \psi_k, \psi^{*}_{k}$ acting on $\Bo$ as follows:
\begin{align*}
\overline{\Psi}(m)=
 \begin{pmatrix}
\dots,\\
 \psi_{k+1}|_{B_m},\\
 \psi_{k}|_{B_m},\\
  \psi_{k-1}|_{B_m},\\
 \dots \end{pmatrix},
 \quad
 \overline{\Psi}^{*}(m)= \begin{pmatrix}
\dots,
 \psi^*_{k+1}|_{B_m},
 \psi^*_{k}|_{B_m},
  \psi^*_{k-1}|_{B_m},
 \dots \end{pmatrix},
\end{align*}
both vectors have   $\psi_k|_{B_m}$ or $ \psi^{*}_{k}|_{B_m}$  entry on the $(-k)$th place.
Denote  as  $L$ the shift operator  $\sum_{k}E_{k,k+1}$.
Then (\ref{vk11}) and  (\ref{vk22})  can be expressed as a relation on  vectors and matrices those entires are operators on $\Bo$:
\begin{align*}
\overline{\Psi}(m)= zL^{m+1}\sum_{p\in \bZ}(-T)^{p} \,{ {\mathcal H}}\,^{(0)}\,    D_p,\quad
\overline{\Psi}^{*}(m)=  z^{-1}L^{m-1}\sum_{p\in\bZ} \,{{\mathcal E}}\,^{(0)}\, T^{p}   D^{(-p)}.
\end{align*}
\begin{example}
For ordinary symmetric functions we  have  $T=L^{-1}$.
\end{example}
\begin{example}
 In \cite{SV}  a Jacobi --Trudi identity is proved for a  large class of  generalized symmetric polynomials, 
 where  generalized Schur polynomials are given by  the formula (\ref{jt1})
with recurrence relation
\begin{align*}
h^{(p+1)}_k= h^{(p)}_{k+1}+a(k)h^{(p)}_k+ b(k)h^{(p)}_{k-1},
\end{align*}
which corresponds to the tridiagonal operator
$T= \sum_{k}  E_{k,k-1}+a({-k})E_{k,k}+b(-k)E_{k,k+1}$.
\end{example}

\begin{example}
Let the elements $h^{(p)}_k$  be defined as
\begin{align}\label{linr}
h^{(p)}_k=\sum_{i=-1}^{l} a_{i}h^{(p-1)}_{k-i} ,
\end{align}
where  $\{a_{-1}=1, a_{0},\dots,  a_{l}\}$ are    fixed complex constants.

This recursion corresponds to  the sum of powers of a shift operator:
$
  T=
  \sum_{i=-1}^{l}a_iL^i.
$ In this particular case it is  possible to decompose formal  distributions   $\Psi(u,m)$ and $\Psi^*(u,m)$ in the form similar to (\ref{psi1}), (\ref{psi11}).

\begin{proposition}\label{prop_lin} Set  $f(u)= \sum_{i=-1}^{l}a_iu^i $.
The recurrence relation (\ref{linr}) with the constants $\{a_{-1}=1, a_{0},\dots,  a_{l}\}$
gives the  generating functions $\Psi(u,m)$
 \begin{align}\label{zuk}
\Psi(u,m)&=u^{m+1}z
 H(u)DE(  -f(u)),\\
 \Psi^*(u,m)&=u^{-m+1}z^{-1}
 E(-u^{-1})DH(  f(u)^{-1}),
\end{align}
 where $H(u)= \sum_{k\in \bZ}   h^{(0)}_{k}u^k$, $E(u)= \sum_{k\in \bZ}   e^{(0)}_{k}u^k$, $DE(u)=\sum_{p\in{\bZ}} D_{p} u^p$, $DH(u)=\sum_{p} D^{(p)} u^p$.
\end{proposition}
\begin{proof}
\begin{align*}
\Psi(u,m)
&
=u^{m+1}z\sum_{p\ge 0} (-1)^p \sum_{s\in \bZ}  \sum_{-1\le i_1\dots i_p\le l} a_{i_1}\dots a_{i_p}   h^{(0)}_{s-i_1-\dots- i_p}  D^{p} u^s\\
&
=u^{m+1}z \sum_{t\in \bZ}   h^{(0)}_{t}u^t\,  \sum_{p\ge 0}  \sum_{-1 \le  i_1, \dots, i_p\le l} a_{i_1}\dots a_{i_p}    u^{i_1+\dots+i_p} (-1)^pD^{p} \\
&
=u^{m+1}z  \sum_{t\in \bZ}   h^{(0)}_{t}u^t\, \sum_{p\ge 0}  \big(-\sum_{i=-1}^{l} a_i   u^{i}\big)^pD^{p},
\end{align*}
and  (\ref{zuk})  follows. Using that
$e^{(p)}_k=\sum_{i=-1}^{l} (-1)^{i-1}a_{i}e^{(p-1)}_{k+i}$,
 the proof of the second equality follows  the same lines.
\end{proof}
\end{example}
\begin{remark}
Recall that for ordinary symmetric functions,
$h^{(p)}_k=h_{p+k}$, $e^{(p)}_k=e_{p-k}$, $D^{(-p)}= D_{h_{p-1}}$, $D_p=D_{e_p}$,  and all together it is a particular case of Proposition \ref{prop_lin}
that corresponds to the sequence  $\{a_{-1}=1, a_{1}=0,\dots,  a_{l}=0\}$. Thus,  Proposition \ref{prop_lin} is a simple alternative combinatorial way to deduce the formulas of
classical vertex operators  (\ref{psi1},  \ref{psi11},  \ref{psi2},  \ref{psi21}).
\end{remark}

\bigskip


\end{document}